\documentclass[a4paper,abstract=true]{scrartcl}
\usepackage[utf8]{inputenc}
\usepackage{amsfonts,amsthm,amssymb,amsmath}
\usepackage[inline]{enumitem}
\usepackage{url}
\usepackage[hidelinks]{hyperref}
\usepackage{color}
\usepackage{tikz}
\usetikzlibrary{calc}

\newtheorem{ctr}{}[section]
\newtheorem{theorem}[ctr]{Theorem}

\newtheorem{proposition}[ctr]{Proposition}
\newtheorem{corollary}[ctr]{Corollary}

\theoremstyle{remark}

\newcommand{\Aut}{\operatorname{Aut}}
\newcommand{\N}{\mathbb{N}}

\newcommand{\ord}{\operatorname{ord}}

\title{Edge Inversions in $(P_k)$-closed Groups}
\author{Kirwin Hampshire, Florian Lehner, and Andrew Wood}
\begin{document}

\maketitle
\begin{abstract}
We construct $(P_2)$-closed groups acting on $T_3$ in which all edge inversions have infinite order. This provides a negative answer to a question posed by Tornier. We also construct a family of $(P_2)$-closed groups for which the smallest order of an edge inversion is an arbitrarily high finite number. 
\end{abstract}

\section{Introduction}

Groups acting on trees play a foundational role in geometric group theory, for instance through Bass–Serre theory and Stallings' theorem about ends of groups. They are also central to the theory of totally disconnected, locally compact (t.d.l.c.) groups as many important examples of such groups come from group actions on trees. 

Tits \cite{tits70} constructed the first examples of simple, non-discrete t.d.l.c.\ groups by showing that the group of automorphisms preserving the bipartition of a regular tree is simple (or trivial). More generally, Tits showed that if a group satisfies a certain independence property, called property $(P)$, then the subgroup generated by edge stabilisers is simple. Several constructions producing groups with property $(P)$ are known, such as the Burger-Mozes universal group construction \cite{burgermozes2000}, and  the box-product construction introduced by Smith~\cite{smith2017} to obtain the first uncountable family of simple, non-discrete t.d.l.c.\ groups. It is worth mentioning that $(P)$-closed groups, that is, closed groups which have property $(P)$ can be fully understood via the local-to-global theory of local action diagrams introduced by Reid and Smith in \cite{reidsmith22}, a theory similar to Bass-Serre theory.

In this short note we focus on groups acting on trees with property $(P_k)$, a generalisation of property $(P)$ which was introduced by Banks, Elder, and Willis in \cite{bew14}, and analogously to property $(P)$ gives rise to simple t.d.l.c.\ groups. 

In \cite{tornier2021}, Tornier gives a construction for $(P_k)$-closed groups along the same lines as the Burger-Mozes universal group construction. Tornier's construction has the property that if the resulting group contains edge inversions (that is, automorphisms swapping the endpoints of an edge), then it contains edge inversions of order $2$. It is not difficult to show that for $(P)$-closed groups this is not a restriction: any $(P)$-closed group which contains an edge inversion must contain an edge inversion of order $2$. However, there are examples of $(P_2)$-closed groups in which the minimal order of an edge inversion is not $2$. Motivated by these observations, Tornier asked whether there is an upper bound on the minimal order of edge inversions in $(P_k)$-closed groups. 

We show that this is not the case by constructing examples of $(P_2)$-closed groups acting on the $3$-regular tree containing edge inversions, but no finite order edge inversions. We also show that there are vertex transitive examples of such groups acting on the $4$-regular tree, and examples of  $(P_2)$-closed groups which have finite order edge inversions but in which the least order of an edge inversion is arbitrarily large.

\section{Notation}

We denote the vertex set of a tree $T$ by $VT$ and the edge set by $ET$. Let $v \in VT$ for some tree $T$. The \emph{ball of radius $n$ centered at $v$} is the set of vertices $B(v,n) = \{w \in VT \mid d(w,v) \leq n\}$, and the \emph{sphere of radius $n$ centered at $v$} is the set of vertices $B(v,n) = \{w \in VT \mid d(w,v) = n\}$.  To extend these notions to edges, let $e \in ET$, let $v$ and $w$ be the two vertices joined by $e$, and let $T_v$ and $T_w$ be the half-trees obtained by removing $e$ which contain $v$ and $w$ respectively; we then define $B(e,n)= B(v,n) \cup B(w,n)$ and $S(e,n)=(S(v,n) \cap T_v) \cup (S(w,n)\cap T_w)$. 

Let $X \subseteq VT$ and $g \in \Aut(T)$, we denote by $g\mid_X$ the restriction of $g$ to the set of vertices $X$. If $g$ setwise stabilises $X$, then this is a permutation of $X$; if it does not stabilise $X$, then this is a map from $X$ to $VT$. We denote the order of $g \in \Aut(T)$ by $\ord(g)$. If $g(X) = X$, then $\ord(g\mid_X)$ denotes the order of the permutation $g$ induces on $X$, that is, the minimal $n$ such that $g^n$ fixes $X$ pointwise.

Let $G \leq \Aut(T)$. The $(P_k)$-closure of $G$, denoted by $G^{(P_k)}$ is defined by
\[
 G^{(P_k)} = \{h \in \Aut(T) \mid \forall v \in VT \, \exists g \in G \colon h\mid _{B(v,k)} = g\mid _{B(v,k)}\},
\]
in other words, the $(P_k)$-closure of $G$ consists of all automorphisms of $T$ whose actions on each $k$-ball coincide with the action of some $g \in G$ (but not necessarily the same $g$) for each $k$-ball. We call a group $(P_k)$-closed, if $G^{(P_k)} = G$. It can be shown that the $(P_k)$-closure of $G$ is the smallest $(P_k)$-closed group containing $G$. If $k = 1$, we write $(P)$-closure instead of $(P_1)$-closure.

\section{Results}

\subsection{A $(P_2)$-closed Group With No Finite Order Inversions}
The following result is well known. Since it is the main motivation for the results in this paper, we provide a short proof for the reader's convenience.

	\begin{proposition}
	\label{prop:1} If a $(P)$-closed group $G$ acting on a tree $T$ contains an edge inversion, then $G$ contains an edge inversion of order 2.
	\begin{proof}
		Let $e=(v,w)$ be an edge in $T$, let $g$ be an inversion of $e$, and let $T_v$ and $T_w$ be the connected components of $T-e$ containing $v$ and $w$ respectively. We define an element $x \in \text{Aut}(T)$ such that \[xu =
		\begin{cases}
			gu & \text{if } u \in VT_v, \\
			g^{-1}u & \text{if } u \in VT_w.
		\end{cases}\]
        Note that $x$ preserves adjacency on $T_v$ and $T_w$ by definition, and since $x(v)=w$ and $x(w)=v$ it preserves adjacency on all of $T$. Bijectivity follows from the fact that $gT_v=T_w$ and $g^{-1}T_w=T_v$, since both $g$ and $g^{-1}$ are bijective. Therefore $x \in \Aut(T)$.
        
        By definition, $x\mid_{T_v}=g\mid_{T_v}$ and $x\mid_{T_w}=g^{-1}\mid_{T_w}$. Since $g(v) = g^{-1}(v) = x(v)$ and $g(w) = g^{-1}(w) = x(w)$ we conclude that $x \in G^{(1)} = G$ since $G$ has property (P). 
        
        Finally, to see that $x^2$ is trivial, consider an arbitrary vertex $u \in VT_v$. Then $x^2u=xgu$ but $gu \in T_w$ so $x(gu)=g^{-1}(gu)=u$ and a symmetric argument holds for any vertex in $T_w$.
	\end{proof}
\end{proposition}

Given that any $(P)$-closed group which contains an edge inversion must also contain an edge inversion of order 2, one might reasonably conjecture that something similar must also be the case for $(P_k)$-closed groups. Corollary \ref{cor:nofiniteorderinversion} below shows that this is not even true for $(P_2)$-closed groups. 

Call an automorphism $g$ of $T_3$ a \emph{good inversion} if it maps $e$ to itself and acts as a cyclic permutation of order $2^n$ on $S(e,n)$ for every $n \in \N$; it is not hard to verify that such an automorphism exists, and that it necessarily swaps the endpoints of $e$.

\begin{theorem}
	\label{thm:2} Let $g \in \text{Aut}(T_3)$ be a good inversion. Then $G=\langle g \rangle^{(P_k)}$ has no elements of order $2$ for $k\geq 2$.
\end{theorem}
	\begin{proof}
		Suppose for a contradiction that $h \in G$ were an element of order $2$. It follows that $h$ does not invert $e$, since any inversion of $e$ in $G$ must be order $4$ on $S(e,1)$. Since $h$ fixes $e$ but $h\neq1_G$, there must be some smallest $n$ such that $\text{ord}(h\mid_{S(e,n)})=2$. We know that $\text{ord}(h\mid_{S(e,m)})\leq\text{ord}(h\mid_{S(e,m+1)})$ and so $\text{ord}(h\mid_{S(e,m)})=2$ for all $m\geq n$. 
        
        Let $u$ be a vertex of $S(e,n-1)$. By the minimality of $n$, we have $hu=u$. Choose $u$ such that $h$ swaps the two neighbours $v,v^{\prime}$ of $u$ in $S(e,n)$; note that this is possible because $\text{ord}(h\mid_{S(e,n)})=2$. Let $w_1,w_2$, and $w_1^{\prime},w_2^{\prime}$ be the neighbours of $v$ and $v'$ in $S(e,n+1)$, respectively. We can see that $\text{ord}(h\mid_{B(u,2)})=2$ since $h$ swaps $v$ and $v^{\prime}$ and must likewise act involutively on $w_1,w_2,w_1^{\prime},w_2^{\prime}$. 
        
        By definition $h\mid_{B(u,2)}=g^p\mid_{B(u,2)}$ for some integer $p$. Therefore, $g^p$ is involutive on $S(e,n+1)\cap B(u,2)$. Since $\langle g \rangle$ induces a single cyclic permutation on all $2^{n+2}$ of the vertices in $S(e,n+1)$ it follows that $p \equiv 2^{n+1} \mod 2^{n+2}$. Hence, $p \equiv 0 \mod  2^{n+1}$ which implies that $g^p$ fixes $S(e,n)$. But this contradicts that $g^p$ has order $2$ on $B(u,2)$.
	\end{proof} 

\begin{corollary}
    \label{cor:nofiniteorderinversion}
    The $(P_k)$-closure of a group generated by a good inversion contains no finite order inversion.
\end{corollary}
\begin{proof}
    A finite order inversion $x$ would have even order $2k$. But then $x^k$ would be an element of order $2$.
\end{proof}

\begin{corollary}
    The $(P_k)$-closure of a group generated by a good inversion is torsion free.
\end{corollary}

\begin{proof}
    If $x$ is a finite order element which is not an inversion, then there is some $n$ so that $x$ pointwise fixes $B(e,n)$, but not $B(e,n+1)$. Consequently, $x$ swaps some pair of vertices in $B(e,n+1)$ and thus the same is true for every odd power of $x$. It follows that $x$ has even order $2k$. But then $x^k$ would be an element of order $2$. 
\end{proof}

\subsection{Inversion Orders in $(P_2)$-closures}

In this section, we take a closer look at the $(P_2)$-closures of groups generated by a single infinite order edge inversion of $T_3$. Our main result is the following theorem which allows us to exactly determine the minimal order of an edge inversion in such a group.

\begin{theorem}
    \label{thm:3}
    Let $g$ be an infinite order inversion of an edge $e$ of $T:=T_3$. Let $N \in \N$ and assume that $\ord(g\mid_{S(e,n)}) = 2^{n+1}$ for every $n < N$.
    \begin{enumerate}
        \item Every inversion $h \in \langle g \rangle^{(P_2)}$ has order at least $2^N$.
        \item If $\ord(g\mid_{S(e,N)}) = 2^N$, then $\langle g \rangle^{(P_2)}$ contains an inversion of order exactly $2^N$.
    \end{enumerate}
\end{theorem}

\begin{proof}
    For the first part, we show by induction on $n$ that $\ord(h\mid_{S(e,n)}) = 2^{n+1}$ for every $n < N$. If $n = 0$, this is trivially true. If $n=1$ and $N>1$, then $g$ has order $4$ on $S(e,1)$, and thus any inversion of $e$ in $\langle g \rangle^{(P_2)}$ has order $4$ on $S(e,1)$.
    
    So we may assume that $n \geq 2$. In this case, $h' := h^{2^{n-1}}$ has order $2$ on $S(e,n-1)$: note that $h'$ pointwise fixes $S(e,n-2)$, but must act non-trivially on $S(e,n-1)$ by induction hypothesis. Since $h' \in \langle g \rangle^{(P_2)}$, for every $u \in S(e,n-2)$ there is some $p$ such that $h' \mid_{B(u,2)} = g^p \mid_{B(u,2)}$. Choose $u$ so that $h'$ swaps the neighbours of $u$ in $S(e,n-2)$. If $\ord(h\mid_{S(e,n)}) = 2^{n}$, then an analogous argument to the proof of Theorem \ref{thm:2} shows that $p \equiv 2^n \mod 2^{n+1}$, contradicting the fact that $h'$ swaps the neighbours of $u$. Therefore $\ord(h\mid_{S(e,n)}) > 2^{n}$, and thus $\ord(h\mid_{S(e,n)}) = 2^{n+1}$ as claimed.

    For the second part, we argue similarly to the proof of Proposition \ref{prop:1}. Note that $g^{2^N}$ pointwise fixes $S(e,n)$, and thus the actions of $g$ and $g^{1-2^N}$ on $S(e,n)$ coincide. 
    
    Enumerate the vertices of $S(e,N-1)$ as $(v_i)_{0 \leq i < 2^{N}}$ and assume that $g(v_i) = v_{i+1}$. Let $T_i$ be the connected component of $T - B(e,N-2)$ containing $v_i$. Define an automorphism $x$ of $T$ by
    \[
        x(v) =
        \begin{cases}
            g^{1-2^N}(v) & \text{if }v \in T_0\\ 
            g(v) & \text{otherwise.}
        \end{cases}
    \]
    As in the proof of Proposition \ref{prop:1}, it is easy to see that this indeed defines an automorphism of $T$. 
    
    Next we note that the restriction of $x$ to a ball of radius 2 around a vertex coincides with the restriction of $g$ or $g^{1-2^N}$. If the ball is entirely contained in $T_0$ or does not intersect $T_0$ at all, this is trivial. Otherwise it follows from the fact that the actions of $g$ and $g^{1-2^N}$ on $S(e,n)$ (and thus also on $B(e,n)$ coincide.

    Finally, we compute the order of $x$. Clearly, $\ord(x\mid_{B(e,N)}) = \ord(g\mid_{B(e,N)}) = 2^N$. If $v \notin B(e,N)$ then $v$ is contained in some $T_i$. We note that $xg^j(v) = g^{j+1}(v)$ if $g^j(v) \notin T_0$ and $xg^j(v) = g^{j+1-2^N}(v)$ if $g^{j}(v) \in T_0$ . Since $g$ maps $T_i$ to $T_{i+1}$ we know that $(g^j(v))_{0 \leq j < 2^N}$ contains exactly one vertex in each $T_i$, and thus $x^{2^N}(v) = g^{2^{N}-2^N}(v) = v$. So $\ord(g\mid_{T - B(e,N)}) \leq 2^N$, and thus the order of $x$ is $2^N$ as claimed.
\end{proof}

As an immediate corollary we obtain the following strengthening of Corollary \ref{cor:nofiniteorderinversion}.

\begin{corollary}
    Let $g$ be a good inversion. Then every inversion in $\langle g \rangle^{(P_2)}$ is a good inversion.
\end{corollary}

We can also conclude that there exist $(P_2)$-closed subgroups of $\Aut(T_3)$ which contain finite order edge inversions of arbitrarily high order, and none of smaller order. 

\begin{corollary}
	\label{Cor:1} For every $m \in \mathbb{N}$ there exists a $g_m \in \text{Aut}(T_3)$ such that $\langle g_m \rangle^{(P_2)}$ contains finite order inversions, but does not contain an inversion of order less than $m$. 
\end{corollary}

\subsection{A Vertex-transitive $(P_2)$-closed Group With No Finite Order Inversions}

The action of the group described in the previous section on $T_3$ has infinitely many orbits. In this section we describe a vertex transitive example of a $(P_2)$-closed group without finite order inversions. In fact, we can even show that this group has no non-trivial torsion elements.

\begin{theorem}
    There is a $(P_2)$-closed, torsion free, vertex transitive group acting on the $4$-regular tree $T_4$.
\end{theorem}

\begin{proof}
To define the group, first define a graph $T_4^+$ by replacing each vertex of $T_4$ by a gadget which consists of a directed 4-cycle of black vertices and two blue vertices which are attached to the first and third, and second and fourth vertices of the directed cycle, respectively, with directed arcs as shown in Figure \ref{fig:gadget}. We then attach each gadget to gadgets corresponding to the neighbouring vertices with red and green arcs in the way shown in Figure \ref{fig:gadget-attach}. Call the resulting coloured, directed graph $\Sigma$.

The group $G$ of colour preserving automorphisms of $\Sigma$ maps each gadget to a gadget because it must map black edges to black edges, and thus has a natural action on $T_4$. If we colour every edge of $T_4$ red or green according to the colour of the edges connecting the corresponding gadgets, and orient red edges accordingly, then the action of $G$ preserves both the colouring and the orientation. Let $H$ be the $(P_2)$-closure of $G$; it can be shown that $H=G$, but we do not need this fact and thus leave the proof to the interested reader.

Clearly, $H$ is $(P_2)$-closed.
Given two gadgets, it is easy to (recursively) construct an element of $G$ which maps one to the other. Hence the action of $G$ on $T_4$ is vertex transitive, and thus $H$ acts vertex transitively as well.

It only remains to show that $G$ has no elements of finite order. Define the \emph{children} of a vertex $v$ to be the two vertices which can be reached from $v$ by traversing a red arc in forward direction, and define the \emph{grandchildren} of $v$ as the children of the children of $v$. Assume that $g \in G$ fixes $v \in VT$. Note that $g$ swaps the children of $v$ if and only if $g$ (in its action on the graph  $T_4^+$) swaps the two blue vertices in the gadget corresponding to $v$. Hence, if $g$ fixes $v$ and swaps its children, then it induces a cyclic permutation of order $4$ on the black vertices in this gadget, and thus also on the blue vertices in the gadgets corresponding to the children of $v$. This implies that $g^2$ swaps the blue vertices in each gadget corresponding to a child of $v$ and thus has order $2$ on the grandchildren of $v$. It follows that $g$ induces a permutation of order $4$ on the grandchildren of $v$ in $T$. Since children and grandchildren of $v$ are contained in $B(v,2)$ and $H$ is the $(P_2)$-closure of $G$, we conclude that if an element $h \in H$ fixes $v$ and swaps its children, then $h$ must induce a permutation of order $4$ on the grandchildren of $v$ as well.

An analogous argument shows that if $h \in H$ swaps the endpoints of an edge (which must be a green edge), then it induces a cyclic permutation of order $4$ on the vertices at distance $1$ from this edge.

Let $h \in H$ be a non-trivial element of finite order in $H$. It is known \cite{halin73} that any automorphism of a tree either fixes a vertex or an edge, or acts like a translation along a double ray. Since translations have infinite order, $h$ either fixes a vertex or swaps the endpoints of an edge. In the former case, there must be at least one vertex which is fixed by $h$, but whose neighbours are not fixed pointwise; the only possibility for this is that $h$ swaps the children of this vertex. In either case, $h$ must have even order $2k$, and therefore $h^k$ has order $2$.

Finally, $h^k$ has finite order, so as before, it must either fix a vertex and swap its children, or swap the endpoints of an edge. But the above discussion shows that in both cases $h^k$ must have at least order $4$, a contradiction.
\end{proof}

\begin{figure}
    \centering
 \begin{tikzpicture}
    [scale=.7,mynode/.style={draw, circle, inner sep=2pt}]

    \node[mynode,fill=black] (1) at (1,0) {};
    \node[mynode,fill=black] (2) at (0,1) {};
    \node[mynode,fill=black] (3) at (-1,0) {};
    \node[mynode,fill=black] (4) at (0,-1) {};
    \node[mynode,fill=blue] (5) at (1,2) {};
    \node[mynode,fill=blue] (6) at (-1,2) {};

    \draw[->, thick] (1) -- (2);
    \draw[->, thick] (2) -- (3);
    \draw[->, thick] (3) -- (4);
    \draw[->, thick] (4) -- (1);

    \draw[->, thick] (6) -- (2);
    \draw[->, thick] (6) -- (4);

    \draw[->, thick] (5) -- (1);
    \draw[->, thick] (5) to [out=180,in=90] (3);

    \draw[dashed](2,0.65) arc[start angle=0, end angle=360, radius=2];
    
 \end{tikzpicture}
    \caption{The gadget with which we replace each vertex in $T_4$.}
    \label{fig:gadget}
\end{figure}
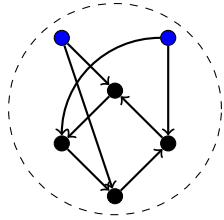

\begin{figure}
    \centering
 \begin{tikzpicture}
    [scale=.7,mynode/.style={draw, circle, inner sep=2pt}]
    
    \node[mynode,fill=black] (1) at (1,0) {};
    \node[mynode,fill=black] (2) at (0,1) {};
    \node[mynode,fill=black] (3) at (-1,0) {};
    \node[mynode,fill=black] (4) at (0,-1) {};
    \node[mynode,fill=blue] (5) at (1,2) {};
    \node[mynode,fill=blue] (6) at (-1,2) {};

    \node[mynode,fill=blue] (7) at (6,2) {};
    \node[mynode,fill=blue] (8) at (4,2) {};

    \node[mynode,fill=black] (9) at (1,4) {};
    \node[mynode,fill=black] (10) at (-1,4) {};
    \node[mynode,fill=black] (11) at (0,4.5) {};
    \node[mynode,fill=black] (12) at (0,3.5) {};

    \node[mynode,fill=blue] (13) at (1,-3) {};
    \node[mynode,fill=blue] (14) at (2.5,-2.3) {};

    \node[mynode,fill=blue] (15) at (-1,-3) {};
    \node[mynode,fill=blue] (16) at (-2.5,-2.3) {};

    \draw[->, thick] (1) -- (2);
    \draw[->, thick] (2) -- (3);
    \draw[->, thick] (3) -- (4);
    \draw[->, thick] (4) -- (1);

    \draw[->, thick] (9) -- (12);
    \draw[->, thick] (10) -- (11);
    \draw[->, thick] (11) -- (9);
    \draw[->, thick] (12) -- (10);

    \draw[->, thick] (6) -- (2);
    \draw[->, thick] (6) -- (4);

    \draw[->, thick] (5) -- (1);
    \draw[->, thick] (5) to [out=180,in=90] (3);

    \draw[->, red, thick] (2) to [out=280,in=140] (14);
    \draw[->, red, thick] (4) -- (13);

    \draw[->, red, thick] (1) to [out=250, in=40] (15);
    \draw[->, red, thick] (3) -- (16);

    \draw[->, red, thick] (9) -- (5);
    \draw[->, red, thick] (10) -- (6);

    \draw[->, green, thick] (7) to [out=150, in=30] (6);
    \draw[->, green, thick] (6) to [out=10, in=150] (8);
    \draw[->, green, thick] (8) to [out=200, in=340] (5);
    \draw[->, green, thick] (5) to [out=310, in=220] (7);

    \draw[dashed](-0.4,-3.8) arc[start angle=0, end angle=100, radius=2,rotate=40];
    \draw[dashed](2.8,-1.8) arc[start angle=80, end angle=180, radius=2,rotate=-40];
    \draw[dashed](-1.5,4) arc[start angle=220, end angle=320, radius=2];
    \draw[dashed](4,3.5) arc[start angle=135, end angle=225, radius=2];
    \draw[dashed](2,0.65) arc[start angle=0, end angle=360, radius=2];
 \end{tikzpicture}
    \caption{A gadget and its arcs to adjacent gadgets in $T_4^+$.}
    \label{fig:gadget-attach}
\end{figure}
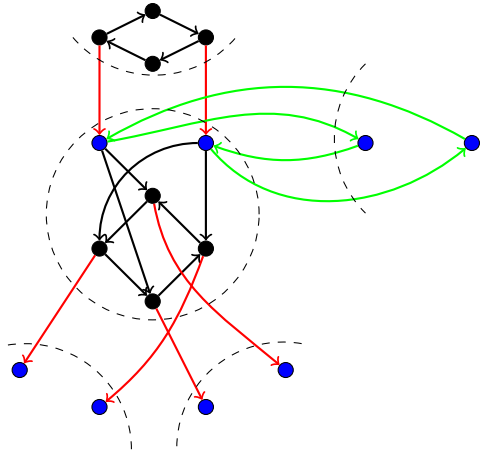

\bibliographystyle{abbrv}
\bibliography{sources.bib}

\end{document}